\documentclass{amsart}
\usepackage{amsfonts}
\usepackage{amssymb}
\usepackage{amsmath}
\usepackage{amsthm, comment}
\usepackage{a4wide, algpseudocode, algorithm}
\usepackage{url,hyperref,mathtools}
\usepackage{graphicx} 
\usepackage{dirtytalk}
\usepackage{float}
\usepackage{mdframed}

\DeclarePairedDelimiter{\floor}{\lfloor}{\rfloor}

\def\BF{\mathbb{F}}
\def\BQ{\mathbb{Q}}
\def\BR{\mathbb{R}}
\def\BZ{\mathbb{Z}}

\def\CD{\mathcal{D}}
\def\CE{\mathcal{E}}
\def\CO{\mathcal{O}}
\def\CP{\mathcal{P}}

\def\fp{\mathfrak{p}}

\def\Cl{\mathrm{Cl}}

\def\Ext{\mathrm{Ext}}
\def\Gal{\mathrm{Gal}}

\def\ed{\varepsilon_d}

\newtheorem{Thm}{Theorem}[section]

\newtheorem{Lem}[Thm]{Lemma}
\newtheorem{Prop}[Thm]{Proposition}
\newtheorem{Conj}[Thm]{Conjecture}
\newtheorem{Heur}[Thm]{Heuristic}
\newtheorem{Rem}[Thm]{Remark}
\newtheorem{Ex}[Thm]{Example}

\begin{document}

\title[Quadratic units and cubic fields]{Quadratic units and cubic fields}

\author{Florian Breuer}
\address{School of Information and Physical Sciences\\
University of Newcastle\\
Callaghan\\
Australia}
\email{florian.breuer@newcastle.edu.au}

\author{James Punch}
\address{School of Science, The University of New South Wales, Canberra, Australia}
\email{j.punch@unsw.edu.au}




\thanks{The authors gratefully acknowledge the use of the National Computational Infrastructure (NCI Australia),  
an NCRIS enabled capability supported by the Australian Government, for providing some of the computational resources that contributed to the results presented in Table \ref{table split by hd} [Project ID: pc06]. 
The first author would like to thank the Alexander-von-Humboldt foundation for support and Bill Allombert for help with Pari/GP}


\begin{abstract}
We investigate Eisenstein discriminants, which are squarefree integers $d \equiv 5 \pmod{8}$ such that the fundamental unit $\varepsilon_d$ of the real quadratic field $K=\mathbb{Q}(\sqrt{d})$ satisfies $\varepsilon_d \equiv 1 \pmod{2\mathcal{O}_K}$. These discriminants are related to a classical question of Eisenstein and have connections to the class groups of orders in quadratic fields as well as to real cubic fields. 
We present numerical computations of Eisenstein discriminants up to $10^{11}$, suggesting that their counting function up to $x$ is approximated by
$\pi_{\mathcal{E}}(x) \approx \frac{1}{3\pi^2}x - 0.024x^{5/6}$.
This supports a conjecture of Stevenhagen while revealing a surprising secondary term,
which is similar to (but subtly different from) the secondary term in the counting function of real cubic fields. 

We include technical details of our computation method, which uses a modified infrastructure approach implemented on GPUs.
\end{abstract}

\maketitle

\section{Introduction}

Let $d \equiv 5 \bmod 8$ be a positive square-free integer. 
Let $K=\BQ(\sqrt{d})$ be the associated real quadratic field with ring of integers
$\CO_K = \BZ[\frac{1+\sqrt{d}}{2}]$ and fundamental unit
$\ed = \frac{x_0 + y_0\sqrt{d}}{2}$.
The prime $2$ is inert in $K/\BQ$.

We denote by $\CO=\BZ[\sqrt{d}]$ the order of conductor 2 in $\CO_K$. The class groups of $\CO_K$ and $\CO$ are related by the following short exact sequence.

\begin{equation}\label{exact sequence}
    1 \rightarrow H_d \longrightarrow \Cl(\CO) \longrightarrow 
    \Cl(\CO_K) \rightarrow 1
\end{equation}

where 
\[
    H_d 
    \cong \left\{ \begin{array}{ll}
         \BZ/3\BZ & \text{if}\quad \ed \equiv 1 \bmod 2\CO_K   \\
         1 & \text{otherwise}.
    \end{array}\right.
\]


We have \cite[Lemma 1.1]{Stevenhagen1996},

\begin{Lem}
    Let $d\equiv 5 \bmod 8$ be square-free. Then the following are equivalent:
    \begin{enumerate}
        \item the fundamental unit satisfies $\ed \equiv 1 \bmod 2\CO_K$;
        \item the equation $x^2 - dy^2 = 4$ has no solution in odd integers $x$ and $y$;
        \item the unit groups $\CO^*$ and $\CO_K^*$ coincide;
        \item $\#\Cl(\CO) = 3\#\Cl(\CO_K)$.
    \end{enumerate}
\end{Lem}

Let us call $d$ satisfying the above equivalent conditions {\em Eisenstein discriminants},
since this goes back to an old question of Eisenstein, who asked for a ``criterion'' for $d$ to have this property; 
one wonders if the above lemma would have satisfied Eisenstein. 
The Eisenstein discriminants form sequence A108160 in \cite{oeis}.

In \cite{Stevenhagen1996} Peter Stevenhagen asks what proportion of discriminants satisfy the Eisenstein property. Since a priori $\ed$ can reduce to any of the three non-zero elements of $\CO_K/2\CO_K$, it is natural to conjecture that this proportion is $1/3$.

Let
\[
\CD := \{d \in\BZ_{>0} \;|\; d\equiv 5 \bmod 8, \; \text{$d$ square-free}\}
\]
and denote by
\[
\CE := \{d \in \CD \;|\; \ed \equiv 1 \bmod 2\CO_K\}
\]
the subset of Eisenstein discriminants.

We define the Eisenstein counting function by
\[
\pi_\CE(x) := \sum_{d\in\CE, \; d \leq x} 1, \qquad x\in\BR.
\]
It is well-known (e.g. \cite[\S18.6]{HardyWright}) that
\[
\pi_\CD(x) := \sum_{d \in \CD, \; d \leq x} 1 = \frac{1}{\pi^2}x + O(x^{1/2 + \varepsilon}),
\]
so Stevenhagen's conjecture can be phrased as 

\begin{Conj}[Stevenhagen]
    We have, as $x \rightarrow \infty$,
    \[
    \pi_\CE(x) \sim \frac{1}{3\pi^2}x.
    \]
\end{Conj}

By relating Eisenstein discriminants $d$ to real cubic fields with discriminant $4d$, and applying the Davenport-Heilbronn Theorem \cite{DH2}, Stevenhagen was able to make the following progress on his conjecture.

\begin{Thm}[Stevenhagen]
\label{Stevenhagen1}
We have, as $x \rightarrow\infty$,
    \begin{enumerate}
        \item $\displaystyle \pi_\CE(x) \ll \frac{1}{2\pi^2}x$
        \item $\displaystyle \pi_\CE(x) \gg x^{1/2}$.
    \end{enumerate}
\end{Thm}

Recent improvements to the Davenport-Heilbronn Theorem allow us to strengthen this result slightly (Theorem \ref{Thm: improved Stevenhagen} below),
but the main purpose of the current article is to explore Stevenhagen's Conjecture numerically.

\section{Link to cubic number fields}

\subsection{Splitting and cubic discriminants}

The link between cubic fields and Eisenstein discriminants is the following \cite[Thm 3.3]{Stevenhagen1996}. 

\begin{Prop}\label{Prop: Stevenhagen}
    Suppose $d\in\CD$. If there exists a real cubic field of discriminant $4d$, then $d\in\CE$. 
    Conversely, suppose $d\in\CE$. Then the following are equivalent
    \begin{enumerate}
        \item $4d$ is the discriminant of a real cubic field;
        \item The exact sequence (\ref{exact sequence}) splits.
    \end{enumerate}
    If the above equivalent conditions hold, then there are in fact $h^*(d)$ non-isomorphic cubic fields of discriminant $4d$, where 
    $h^*(d) = \#\Cl(\CO_K)[3]$ is the order of the $3$-torsion subgroup of the class group $\Cl(\CO_K)$.
\end{Prop}

\subsection{Counting cubic discriminants}

Denote by
\begin{align*}
\pi_c(x) &= \#\{\text{non-isomorphic cubic fields of discriminant $4d$, where $d\in\CD$ and $d\leq x$} \} \\
& = \sum_{d\in\CD,\; d\leq x} m(4d)
\end{align*}
the counting function for the relevant cubic fields,
where $m(4d)$ denotes the multiplicity of the cubic discriminant $4d$, i.e. the number of non-isomorphic cubic fields of discriminant $4d$.
By Proposition \ref{Prop: Stevenhagen}, $\pi_c$ should have a similar, but not identical, behaviour to $\pi_\CE$. 

Happily, $\pi_c$ is quite well understood. We have \cite{BST2013, BTTT2024}

\begin{Thm}[Bhargava, Shankar, Taniguchi, Thorner, Tsimmerman]
\label{BSTTT1}
    \[
    \pi_c(x) = \sum_{d\in\CD,\; d\leq x} m(4d) = 
    \frac{1}{3\pi^2}x + C_{5/6} x^{5/6} + O(x^{2/3 + \varepsilon}),
    \]
    where
    $C_{5/6} \approx -0.03761.$
\end{Thm}

\begin{proof}
    We use \cite[Theorem 1.3]{BTTT2024}. In their notation, we have 
    \begin{align*}
        \pi_c(x) &= N_3^+(4x, \Sigma) \\
        &= C^+(\Sigma)\frac{1}{12\zeta(3)}4x + K^+(\Sigma)\frac{4\zeta(1/3)}{5\Gamma(2/3)^3\zeta(5/3)}(4x)^{5/6} + O_{\epsilon}(x^{2/3+\epsilon}),
    \end{align*}
    where $C^+(\Sigma) = \prod_p C_p$ and $K^+(\Sigma) = \prod_p K_p$ are computed next.
    
    The local conditions $\Sigma = (\Sigma_p)_p$ are the following. Since $d$ is square-free and $d\equiv 5 \bmod 8$, real cubic fields with discriminant $4d$ are precisely those at which $2$ is the only prime which is totally ramified, by \cite[Lemma 2.1]{Stevenhagen1996}. Thus, by the helpful Table 2 in \cite{BTTT2024}, we get
    \[
    C_2 = \frac{1}{2^2+2+1} = \frac{1}{7}
    \]
    and
    \[
    K_2 = \frac{1+2^{-1/3}}{2^2}\cdot \frac{1-2^{-1/3}}{(1-2^{-5/3})(1+2^{-1})} = \frac{1-2^{-2/3}}{6(1-2^{-5/3})}.
    \]
    At primes $p\geq 3$ we require that $p$ not be totally ramified, giving
    \[
    C_p = 1 - \frac{1}{p^2+p+1}
    \]
    and
    \[
    K_p = 1 - \frac{1+p^{-1/3}}{p^2}\cdot \frac{1-p^{-1/3}}{(1-p^{-5/3})(1+p^{-1})}.
    \]
    Now we compute
    \begin{align*}
        C_1 &= C^+(\Sigma)\frac{1}{3\zeta(3)} = \frac{1}{21} 
        \frac{\prod_p\left(1-\frac{1}{p^2+p+1}\right)}{\left(1 - \frac{1}{2^2+2+1}\right)}
        \prod_p\left(1 - \frac{1}{p^3}\right)\\
        &= \frac{1}{18}\prod_p\left(1-\frac{1}{p^2}\right) = \frac{1}{18\zeta(2)}\\
        &= \frac{1}{3\pi^2}
    \end{align*}
    and, numerically,
    \begin{align*}
        C_{5/6} &= K^*(\Sigma)\frac{4\zeta(1/3)}{5\Gamma(2/3)^3\zeta(5/3)}\cdot 4^{5/6} 
        = \frac{4^{11/6}\zeta(1/3)}{5\Gamma(2/3)^3\zeta(5/3)}\prod_p K_p\\
        &\approx -0.03761.
    \end{align*}
\end{proof}

This is an improvement over the original Davenport-Heilbronn Theorem \cite{DH2}, which only produced the linear term.

A closely related result is

\begin{Thm}[Bhargava, Shankar, Taniguchi, Thorner, Tsimmerman]
\label{BSTTT2}
    \[
    \sum_{d\in\CD,\; d\leq x}\frac{h^*(d)-1}{2} = 
    \sum_{d\in\mathcal{D},\; d\leq x} m(d) = 
    \frac{1}{6\pi^2}x + C^*_{5/6} x^{5/6} + O(x^{2/3+\varepsilon}),
    \]
    where $C^*_{5/6} \approx -0.0386.$
\end{Thm}

\begin{proof}
    The first identity follows as the number of cubic fields of fundamental discriminant $d$ equals $(h^*(d)-1)/2$, see \cite[Section 8.1]{BST2013}. The counting function for such fields is again given in  \cite[Theorem 1.3]{BTTT2024}. This time, the relevant local condition is that $2$ is partially split and the remaining primes are not totally ramified, by Lemma \ref{partial split} below. The coefficients are computed using Table 2 in \cite{BTTT2024} as in the proof of Theorem \ref{BSTTT1}.
\end{proof}

\begin{Lem}\label{partial split}
    Let $F/\BQ$ be a totally real cubic field with square-free discriminant $d$. Then $2$ is partially split in $F/\BQ$ if and only if $d\equiv 5 \bmod 8$.
\end{Lem}

\begin{proof}
    First, $2$ is unramified if and only if $d$ is odd.
    It is known (see e.g. \cite{Hasse1930}) that if an odd cubic discriminant $d$ is square-free then $d \equiv 1 \bmod 4$ and the Galois closure of $F$ is $L=F(\sqrt{d})$ and has Galois group $\Gal(L/\BQ)\cong S_3$. 
    Since this is not cyclic, $2$ cannot remain inert in $L$, therefore 2 splits into either 2, 3 or 6 primes $\fp$ in $L/\BQ$ of residue degrees $f(\fp, L/\BQ) = 3, 2$ or $1$, respectively.

    If the residue degrees are $1$ or $3$, then $2$ splits in the quadratic extension $\BQ(\sqrt{d})/\BQ$, hence $d \equiv 1 \bmod 8$.

    Therefore, $d\equiv 5 \bmod 8$ if and only if 2 splits into three primes of residue degree 2 in $L/\BQ$. It remains to show that in this case $2$ is partially split in $F/\BQ$.

    The decomposition groups $D(\fp, L/\BQ)$ have order 2, and only one of them can equal $\Gal(L/F)$. Let $\fp$ be one of the two primes above 2 for which $D(\fp, L/\BQ)=\{1, \sigma\} \neq \Gal(L/F)$. Then $\sigma(\fp\cap\CO_F) \neq \fp\cap\CO_F$ and so $f(\fp, L/F) = 1$ and hence $f(\fp\cap\CO_F, F/\BQ)=2$. It follows that 2 is partially split in $F/\BQ$.
\end{proof}

\begin{Rem}
    As a sanity check, we computed cubic discriminants  up to $10^8$ using Pari/GP \cite{PARI2} and found that the results closely matched Theorems \ref{BSTTT1} and \ref{BSTTT2}. In particular, the given approximate values of the coefficients $C_{5/6}$ and $C^*_{5/6}$ appear to be correct.
\end{Rem}



\subsection{A random walk?}

Let us define the difference 
\[
\Delta(x) = \pi_\CE(x) - \pi_c(x), \qquad x\in\BR.
\]

Then we have
\[
\Delta(x) = \sum_{d\in\CE, \; d \leq x} \delta(d),
\]
where, for each $d\in\CE$
\begin{align*}
\delta(d) & = 1 - m(4d) \\
& = \left\{\begin{array}{ll}
    1 & \text{if (\ref{exact sequence}) does not split} \\
    1-h^*(d) & \text{if (\ref{exact sequence}) splits.} 
\end{array}\right. 
\end{align*}

If $h^*(d)=1$, then (\ref{exact sequence}) splits and $\delta(d)=0$. The more interesting cases are when $h^*(d) \geq 3$.

\begin{Ex}
    When $d=1901$, we find $d\in\CE$, $h^*(d)=3$ and (\ref{exact sequence}) does not split, so $m(4d)=0$ and $\delta(1901) = 1$. Here, there are no cubic fields to detect $1901\in\CE$ and $\pi_c$ undercounts $\pi_\CE$ by~1.

    When $d=7053$, we find that $d\in\CE$, $h^*(d)=3$ and (\ref{exact sequence}) splits. In this case, $m(4d)=3$ and $\delta(7053)=-2$. 
    Here, there are three cubic fields corresponding to $7053\in\CE$, so $\pi_c$ overcounts $\pi_\CE$ by~2.
\end{Ex}

When exactly does (\ref{exact sequence}) split? 
Up to equivalence, all possible extensions of the group $\Cl(\CO_K)$ by $H_d\cong\BZ/3\BZ$ are in bijection with the first extension group 
$\Ext^1(\Cl(\CO_K), H_d)$, which we can calculate (e.g.  \cite[Chapter 3]{Weibel94}),  
\[
\Ext^1(\Cl(\CO_K), \BZ/3\BZ) \cong \Cl(\CO_K)[3].
\]
The sequence (\ref{exact sequence}) splits if and only if it corresponds to the identity element in $\Ext^1(\Cl(\CO_K), H_d)$.

A priori, we might expect the exact sequence (\ref{exact sequence}) to correspond to a randomly chosen element of $\Ext^1(\Cl(\CO_K), H_d)\cong \Cl(\CO_K)[3]$, suggesting:

\begin{Heur}\label{splitting}
    For any $d\in\CE$, the probability that (\ref{exact sequence}) splits is $1/h^*(d)$.
\end{Heur}

Given our heuristic, we might model $\Delta(x)$ as a random walk, which at each step either increases by 1 with probability $(h^*(d)-1)/h^*(d)$ or decreases by $h^*(d) - 1$ with probability $1/h^*(d)$. The expected value of each step is thus 0, and the expected variance would be something like $O(x^{1/2 + \varepsilon})$.

Note that, for most $d\in\CD$, $h^*(d)=1$ in which case $\delta(d)=0$. The Cohen-Lenstra-Martinet Heuristics \cite{CL84} predict how often each value of $h^*(d)$ occurs. By Theorem \ref{BSTTT2}, the average value of $h^*(d)$ is $4/3$ and this is notably one of the only cases of the Cohen-Lenstra-Martinet Heuristics that has been proved.








\subsection{Unconditional results}

Before presenting our numerical results, we obtain some upper and lower bounds on $\Delta(x)$, which allow for slight improvements in Stevenhagen's Theorem \ref{Stevenhagen1}.

\begin{Thm}\label{Thm: improved Stevenhagen}
    We have the following asymptotic upper and lower bounds on $\pi_{\CE}(x)$.
    \begin{enumerate}
        \item $\displaystyle \pi_\CE(x) \geq (C_{5/6} - 2C^*_{5/6})x^{5/6} + O(x^{2/3 + \varepsilon})$. 

        \item $\displaystyle \pi_\CE(x) \leq  \frac{1}{2\pi^2}x + \left(C_{5/6} + C^*_{5/6}\right)x^{5/6} + O(x^{2/3 + \varepsilon}).$
        
    \end{enumerate}
\end{Thm}

Since $C_{5/6} - 2C^*_{5/6} \approx 0.0396 > 0$ and $C_{5/6}+C^*_{5/6} \approx -0.0761 < 0$, this is an improvement on the lower and upper bounds in Theorem~\ref{Stevenhagen1}.

\begin{proof}

First,
\begin{align*}
    \Delta(x) & \geq \sum_{d\in\CE, \; d \leq x}(1-h^*(d)) 
    \geq \sum_{d\in\CD, \; d \leq x}(1-h^*(d)) \\
    & = -\frac{1}{3\pi^2}x - 2C^*_{5/6}x^{5/6} + O(x^{2/3+\varepsilon}) \\
    \Rightarrow \; \pi_\CE(x) & = \pi_c(x) + \Delta(x) \\
    & \geq (C_{5/6} - 2C^*_{5/6})x^{5/6} + O(x^{2/3 + \varepsilon}).
\end{align*}

For an upper bound on $\Delta(x)$, note that $\delta(d)=1 \Rightarrow h^*(d) \geq 3$, so 
\begin{align*}
\Delta(x) &\leq \sum_{d\in\CE,\; d\leq x} \frac{h^*(d)-1}{2} \\
& \leq \sum_{d\in\CD,\; d\leq x} \frac{h^*(d)-1}{2} \\
& = \frac{1}{6\pi^2}x + C^*_{5/6}x^{5/6} + O(x^{2/3 + \varepsilon}),
\end{align*}
which gives us
\[
\pi_{\CE}(x) \leq \pi_c(x) + \Delta(x) = 
\frac{1}{2\pi^2}x + \left(C_{5/6} + C^*_{5/6}\right)x^{5/6} + O(x^{2/3 + \varepsilon}).
\]

\end{proof}

\section{Numerical results}

\subsection{Eisenstein discriminants and the splitting of (\ref{exact sequence})}

We computed $\pi_\CE(x)$ for $x\leq 10^{11}$ by adapting Shanks's infrastructure method to compute $\varepsilon_d \bmod 2\CO_K$. The details of the implementation, which runs on a GPU, are described in the Appendix.

\medskip

\begin{mdframed}
We find that $\pi_\CE(x)$ is closely approximated by
\begin{equation}
    \label{eq:main result}
    \pi_{\CE}(x) \approx \frac{1}{3\pi^2}x - 0.024 x^{5/6}.
\end{equation}
\end{mdframed}

Note that the coefficient of $x^{5/6}$ differs from $C_{5/6}$ in Theorem \ref{BSTTT1}, and consequently

\begin{equation}
    \label{eq:main resul 2}
    \Delta(x) \approx (-0.024 - C_{5/6})x^{5/6} \approx 0.013\, x^{5/6}.
\end{equation}

In particular, the expectation that $\Delta(x)$ behaves like an unbiased random walk does not bear out.
This means that the exact sequence (\ref{exact sequence}) splits slightly less often than predicted by Heuristic~\ref{splitting}, although the splitting probability should tend to $1/h^*(d)$ asymptotically since $\Delta(x)$ appears to lack a linear term.


One might expect the splitting of (\ref{exact sequence}) to depend on the $3$-rank of $\Cl(\CO_K)$. We computed $\Cl(\CO_K)$ and $m(4d)$ for $d\in \CE\cap I$ for certain subintervals $I$ of length $10^8$ in $[0,10^{11}]$ using Pari/GP; the results are shown in Table \ref{table split by hd}.


\begin{table}
{\small 
    \centering
\begin{tabular}{|r||r|rr||r|rr|}\hline
     $\CE\cap I$
     & \multicolumn{3}{|c||}
     {$\CE\cap [0, 10^8]$}
     &  \multicolumn{3}{|c|}{$\CE\cap [10^{9}-10^8,10^{9}]$}\\ \hline
     $h^*(d)$ & Total & Split & & Total & Split &  \\ \hline
     1 
     & 2 790 560 & 2 790 560 & (100\%)
     & 2 810 693  & 2 810 693 & (100\%) \\
     3 
     & 464 866 & 136 878 & (29.44\%)
     & 494 628 & 154 116 & (31.16\%) \\ 
     9 
     & 4 241 & 205 & (4.83\%)
     & 5 512 & 404 & (7.33\%) \\
     27 
     & 1 & 0 & (0.00\%)
     & 2 & 1 & (50.00\%) \\ \hline
\end{tabular}

\medskip

\begin{tabular}{|r||r|rr||r|rr|}\hline
    $\CE\cap I$
    & \multicolumn{3}{|c|}{$\CE\cap [10^{10}-10^8,10^{10}]$} 
    & \multicolumn{3}{|c|}{$\CE\cap [10^{11}-10^8,10^{11}]$}
    \\ \hline
    $h^*(d)$ & Total & Split & 
    & Total & Split &
    \\ \hline
    1 & 2 821 015 & 2 821 015 & (100\%) 
    & 2 824 295 & 2 824 295 & (100\%) \\
    3 & 506 849 & 161 395 & (31.84\%) 
    & 514 475 & 166 316 & (32.33\%) \\ 
    9 &  6 359 & 553 & (8.70\%) 
    & 6 729 & 641 & (9.53\%)\\
    27 & 4 & 0 & (0.00\%) 
    & 4 & 0 & (0.00\%)\\ \hline
\end{tabular}
    \caption{Splitting of (\ref{exact sequence}) for Eisenstein discriminants in the given intervals}
    \label{table split by hd}
    }
\end{table}

We see that when $h^*(d) = 3$, the sequence splits slightly less than $1/3$ of the time, with a larger shortfall (4.83\%, 7.33\%, 8.70\% and 9.53\%, respectively, instead of 11.11\%) when $h^*(d) = 9$. There were too few instances of $h^*(d)=27$ to draw any conclusions. We see that the observed proportions approach the expected proportions as $d$ increases.


An interpretation of (\ref{eq:main result}) is that the  ``probability'' 
that a given $d\in\CD$ also lies in $\CE$ is approximately 
\begin{equation}\label{eq:prob E|D}
P(d\in\CE | d\in\CD) \approx \frac{1}{3} - \frac{0.201}{d^{1/6}},
\end{equation}
for then
\begin{align*}
\pi_\CE(x) & = \int_5^x P(\floor{t}\in\CE)dt 
= \int_5^x P(\lfloor t\rfloor \in\CE|\lfloor t\rfloor\in\CD)P(\lfloor t\rfloor\in\CD)dt \\
& \approx \int_5^x\left(\frac{1}{3} - \frac{0.201}{t^{1/6}}\right)\cdot\frac{1}{\pi^2} dt
\approx \frac{1}{3\pi^2}x - 0.024 x^{5/6}.
\end{align*}

We remark that the data in Table \ref{table split by hd} is consistent with
\[
P(\text{(\ref{exact sequence}) splits} \;|\; d\in\CE\wedge h^*(d)=3) \approx \frac{1}{3} - \frac{0.68}{d^{1/6}}
\]
and
\[
P(\text{(\ref{exact sequence}) splits} \;|\; d\in\CE\wedge h^*(d)=9) \approx \frac{1}{9} - \frac{1.08}{d^{1/6}}.
\]







\subsection{Prime subsequence}

The distribution of prime numbers in $\CE$ was investigated in \cite{BSC}. If we denote by $\CP$ the set of prime numbers and 
\[
\pi_{\CE\cap\CP}(x) = \sum_{p\in\CE\cap\CP, \; p\leq x}1,
\]
then computations in  \cite{BSC} for $x \leq 10^{11}$  show that
\begin{equation}\label{eq:pdata}
   \pi_{\CE\cap\CP}(x) \approx \frac{1}{12}\pi(x) - 0.037\int_2^x\frac{dt}{t^{1/6}\log t}. 
\end{equation}

Again, if we write
\begin{align*}
    \pi_{\CE\cap\CP}(x) & = \int_2^x P(\floor{t}\in\CE|\floor{t}\in\CD\cap\CP)P(\floor{t}\in\CD\cap\CP)dt \\
    & = \int_2^x P(\floor{t}\in\CE|\floor{t}\in\CD\cap\CP)\cdot \frac{1}{4\log t} dt,
\end{align*}
then we conclude that the ``probability'' that a prime $p\in\CD$ is Eisenstein is
\begin{equation}\label{eq: prob E|PD}
    P(p\in\CE|p\in\CD\cap\CP) \approx \frac{1}{3} - \frac{0.148}{p^{1/6}}.
\end{equation}

So primes $p\in\CD$ are slightly more likely to be Eisenstein than are composite $d\in\CD$.

\subsection{Discussion}

These results are quite surprising. The distributions of cubic discrimants and of  Eisenstein discriminants are weakly related by Proposition \ref{Prop: Stevenhagen}, but this relation is not enough to force $\pi_c(x)$ and $\pi_\CE(x)$ to have the same linear term, hence the relatively weak  Theorem \ref{Thm: improved Stevenhagen}.

The numerical results suggest that they do indeed have the same linear term and this is what we expect heuristically. Mysteriously, both counting functions also have $x^{5/6}$ terms, although this time the coefficients differ. Can there really be two different ``mechanisms'' producing $x^{5/6}$ error terms? Finally, the case of primes in $\CE$ again has a similar error term, but with a different coefficient. 
Clearly, more remains to be discovered.


\appendix
\section{}
The following material includes summarised parts of \cite{Punch2024} (the second author's Honours thesis).

For each value of $d\in\CD$, we describe how to compute $\varepsilon_d\bmod 2\CO_K$. Our technique has been adapted slightly from \cite[Section 7.4]{JacobsonWilliams}, which computes the (non-reduced) fundamental unit. Our strategy is to reduce intermediate values of $\theta$ (the generators for principal ideals in the cycle of reduced ideals) modulo $2\CO_K$ at each step. The original infrastructure technique is summarised below.

Given a principal $\CO_K$-ideal $I=(\theta_{j+1})$ expressed as the $\BZ$-module $[\frac{Q_j}{2},\frac{P_j+\sqrt{d}}{2}]$, we define $\rho(I)$ to be the ideal $(\theta_{j+2})$, which is the $\BZ$-module $[\frac{Q_{j+1}}{2},\frac{P_{j+1}+\sqrt{d}}{2}]$, where $$q_j=\floor{\frac{P_j+\sqrt{d}}{Q_j}},P_{j+1}=q_j Q_j-P_j, Q_{j+1}=\frac{d-P_{j+1}^2}{Q_j}, \text{ and }\theta_{j+2}=\frac{P_{j+1}+\sqrt{d}}{Q_j}\theta_{j+1}.$$ The action of $\rho$ is usually called a \say{reduction}, because applying it to a primitive ideal sufficiently many times will yield a reduced primitive ideal. A binary operation that yields an ideal equivalent to the product of two input ideals will be denoted by $*$ (note that there are several operations that do this). Infrastructure is a baby-step giant-step technique (due to Shanks, which he describes in \cite{shanks1969}), with $\rho$ and $*$ as baby steps and giant steps, respectively. The reader should refer to Algorithm \ref{alg:Infra} for the full description.

\begin{algorithm}
\caption{Infrastructure}\label{alg:Infra}
\begin{algorithmic}
\State \textbf{Input:} a value $d\in\CD$.
\State \textbf{Output:} the fundamental unit $\varepsilon$ in $\BQ(\sqrt{d})$.
\State Put $j\gets 1,\theta_1\gets 1$. \Comment{$(\theta_1)$ can be thought of as the $\BZ$-module $[\frac{2}{2},\frac{1+\sqrt{d}}{2}]$.}
\State Initialize a dictionary of ideals, and add the key-value pair $(Q_0=2,P_0=1):\theta_1=1$.
\While{$\log\theta_{j}<{\sqrt[^4]{d}}$}
    \State Compute $(\theta_{j+1})=\rho((\theta_{j}))$ and its corresponding $\BZ$-module $[\frac{Q_{j}}{2},\frac{P_{j}+\sqrt{d}}{2}]$.
    \State Add $(Q_j,P_j):\theta_{j+1}$ to the dictionary.
    \If{$Q_j=Q_{j-1}$ or $P_j=P_{j-1}$}
        \State Compute $\varepsilon$ using \cite[Theorem 3.13]{JacobsonWilliams}.
        \State \Return $\varepsilon$.
    \EndIf
    \State $j \gets j+1$.
\EndWhile
\State Put $(\mu_1)\gets(\theta_j),k\gets 2$.
\State Compute two more ideals, $(\theta_{j+1})$ and $(\theta_{j+2})$, and their corresponding $\BZ$-modules.
\State Add $(Q_j,P_j):\theta_{j+1}$ and $(Q_{j+1},P_{j+1}):\theta_{j+2}$ to the dictionary.
\While{true}
    \State Put $(\mu_k)\gets(\mu_1)*(\mu_{k-1}')$.
    \State Apply $\rho$ to $(\mu_k)$ until it is a reduced ideal, say $(\mu_k')$, and let $q_k$ and $p_k$ be the values of $Q$ and $P$ when we write $(\mu_k')$ as a $\BZ$-module.
    \If{$(q_k,p_k)$ is in the dictionary}
        \State Set $\theta$ to be the value corresponding to the key $(q_k,p_k)$.
        \State Put $\varepsilon=\frac{\mu_k'}{\theta}$.
        \State \Return $\varepsilon$.
    \EndIf
    \State $k\gets k+1$.
\EndWhile

\end{algorithmic}
\end{algorithm}

Now, we describe how to modify this procedure. During the baby-step stage of the algorithm, we just reduce each generator modulo $2\CO_K$. The giant step involves composition and reduction on two principal ideals, say $\mathfrak{b}_1=(\theta')$ and $\mathfrak{b}_2=(\theta'')$. The output is a reduced principal ideal $\mathfrak{b}=(\theta)$ and a factor $\frac{1}{\gamma}$ such that $\mathfrak{b}=\frac{1}{\gamma}\mathfrak{b}_1\mathfrak{b}_2=(\frac{1}{\gamma}\theta'\theta'')$. We find $\theta\pmod{2\CO_K}$ by finding $\frac{1}{\gamma}\times\theta'\times\theta''\pmod{2\CO_K}$, so we need some value for $\frac{1}{\gamma}\pmod{2\CO_K}$. While we are guaranteed that $\theta',\theta'',\theta\in\CO_K$, we know only that $\frac{1}{\gamma}\in K$.

However, we can fix this in the following way. In the cycle of reduced principal ideals $\{ \mathfrak{a}_1=(1),\mathfrak{a}_2,\mathfrak{a}_3,\dots\}$, we have that $\theta_j=\frac{G_{j-2}+B_{j-2}\sqrt{d}}{2}$ is a generator for $\mathfrak{a}_j$ (see \cite[Section 3.1]{JacobsonWilliams} for the definitions of $G_j$, $A_j$ and $B_j$).
\begin{Thm} \label{theoremthetanotin2OK}
    For every $j\geq 1$, $\theta_j$ is not in $2\CO_K$.
\end{Thm}
\emph{Proof.}
\begin{itemize}
    \item When $j=1$, $\theta_j=1$.
    \item When $j=2$, $\theta_j=\frac{G_0+B_0\sqrt{d}}{2}$, which is not in $2\CO_K$ because $B_0=1$.
    \item When $j\geq 3$, first observe that, due to \cite[Equation 3.4]{JacobsonWilliams}, $\gcd(G_{j-2},B_{j-2})$ is either 1 or 2. If it is 1, then $\theta_j=\frac{G_{j-2}+B_{j-2}\sqrt{d}}{2}$ must be outside $2\CO_K$. If it is 2, we need to show that $G_{j-2}-B_{j-2}$ (which equals $2A_{j-2}-2B_{j-2}$ by definition of $G_{j-2}$) is $2\pmod{4}$.
    $B_{j-2}$ is even by assumption, and $A_{j-2}$ is coprime to it by \cite[Equation 3.9]{JacobsonWilliams} (and therefore odd), so we are done. \qed
\end{itemize}

It follows that $\theta',\theta'',\theta\in \CO_K \backslash 2\CO_K$. Now, let $\CO_{K,2}$ be the localisation of $\CO_K$ at the prime ideal $2\CO_K$. We know that $\frac{1}{\gamma}$ is $\frac{\theta}{\theta'\theta''}$, and that neither the numerator nor denominator is in $2\CO_K$ (by Theorem \ref{theoremthetanotin2OK} and primality of $2\CO_K$). It follows that $\frac{1}{\gamma}\in \CO_{K,2}$. Putting $B=\CO_K/2\CO_K$ with the canonical reduction map $\CO_K\rightarrow B$, the universal property provides a unique way of extending this map to all of $\CO_{K,2}$. Then, if $\frac{1}{\gamma}=\frac{\alpha}{\beta}$ for some $\alpha\in \CO_K,\beta\in\CO_K\backslash 2\CO_K$, it reduces to $(\alpha+2\CO_K)(\beta+2\CO_K)^{-1}$.

Finally, since we only ever multiply or divide generators, it is sufficient to work in  $(\CO_K/2\CO_K)^*\cong(\BF_4)^*\cong \BZ_3$. Instead of multiplication in $\CO_K/2\CO_K$, it is equivalent to do addition in $\BZ_3$ (more easily achieved on a computer).

Ultimately, this means that the algorithm from \cite[Section 7.4]{JacobsonWilliams} can be performed as usual, except that we substitute elements of $\BZ_3$ for ideal generators from $\CO_K$. Then the result is $0\in \BZ_3$ iff $\varepsilon_d\in 1+2\CO_K$.

This procedure was implemented using the PyOpenCL library for Python. The modified infrastructure technique was implemented in C; for a given value of $d$, it calculated the reduced fundamental unit and returned an element of $\BZ_3$. This C routine was executed on a GPU, with each kernel fed a different value of $d\in \CD$. The outputs were returned to Python, which collated the results. The specific hardware used was a 12th-Gen Intel Core i5-1235U (CPU), and an Intel Iris $\text{X}^\text{e}$ Graphics card (GPU).

Additionally, our implementation uses a nonstandard technique for storing the baby step list. When there is enough memory available, it is best to implement a dictionary as a hash table because lookup is faster in this setting. But on a GPU, each kernel has limited memory, so this is not feasible. Instead, for each ideal, we add the tuple $(Q,P,\theta)$ to an unsorted list. Then, we add $(Q,P)$ to a Bloom filter (see \cite[Method 2]{Bloom} for a description of Bloom filters). Usually, it can quickly answer in the negative when asked if the dictionary contains a given ideal. Rarely, it will give false positives, in which case we must search the whole list. Importantly, it can never give false negatives.

The code that was used for the calculations can be seen at \url{https://github.com/JP1681/Fundamental-units-modulo-2O_K}. The authors acknowledge the help of GPT-4 in implementing the Bloom filter and related search functions.

Earlier, we mentioned that there are several ways of carrying out $*$. One such algorithm is known as NUCOMP, created by Shanks and described by him in \cite{Shanks1988}. Many versions of the algorithm exist (see \cite[Appendix A.1]{JacobsonWilliams}, \cite[Section 3]{JacobsonScheidlerWilliams},\cite[Section 10]{vdP2003},\cite[Section 2]{JacobsonvdP2002}). However, the authors could not get them to perform correctly. For example, they would either give invalid ideals, require more reduction steps than claimed, or produce an incorrect value of $\gamma$.

Ultimately, the authors were able to produce a working version of NUCOMP. Ours is a \say{hybrid}, based on \cite[Section 10]{vdP2003} and \cite[Section 2]{JacobsonvdP2002}, and is described here in full. Our algorithm is the same as in \cite{JacobsonvdP2002}, except that the value of $\gamma$ is calculated using \cite{vdP2003}.

As in those papers, we use binary quadratic forms instead of ideals, and we have two algorithms: NUCOMP, for composing different forms; and its special case NUDUPL, for composing a form with itself. We also describe a third algorithm, which mainly serves to convert between forms and ideals.

For NUCOMP (Algorithm \ref{alg:NUCOMP}), it is sensible to precompute $L=\lvert d \rvert^\frac{1}{4}$. Except where noted, all divisions are done \textbf{without} remainder.

\begin{algorithm}
\caption{NUCOMP}\label{alg:NUCOMP}
\begin{algorithmic}
\State \textbf{Input:} two binary quadratic forms, $\varphi_1=(u_1,v_1,w_1)$ and $\varphi_2=(u_2,v_2,w_2)$, both with discriminant $d$, and the value $L=\lvert d \rvert^\frac{1}{4}$.
\State \textbf{Output:} a near-reduced binary quadratic form $\varphi_3=(u_3,v_3,w_3)$, whose discriminant is $d$, and the value $\gamma\in K$ such that $\varphi_3=\frac{1}{\gamma}\varphi_1 \varphi_2$.
\\
\If {$w_1<w_2$}
    \State Swap $\varphi_1$ and $\varphi_2$.
\EndIf
\State Put $s\gets\frac{v_1+v_2}{2}$, $m\gets v_2-s$.
\State Use the extended Euclidean algorithm to find $(b,c,F)$ such that $bu_2+cu_1=F=\gcd(u_1,u_2)$.
\If {$F\vert s$}
    \State Put $G\gets F, A_x\gets G, B_x\gets mb, B_y\gets \frac{u_1}{G},C_y\gets \frac{u_2}{G}, D_y\gets\frac{s}{G}$.
\Else
    \State Use the extended Euclidean algorithm to find $(x,y,G)$ such that $xF+ys=G=\gcd(F,s)$.
    \State Put $H\gets \frac{F}{G}, B_y\gets \frac{u_1}{G}, C_y\gets \frac{u_2}{G},  D_y\gets\frac{s}{G}$.
    \State Put $l\gets y(bw_1+cw_2)\mod H,B_x\gets\frac{bm+lB_y}{H}$.
\EndIf
\State Put $b_x\gets B_x\mod B_y, b_y\gets B_y, x\gets 1, y\gets 0, z\gets 0$.
\While{$\lvert b_y \rvert > L$ and $b_x\neq 0$}
    \State Divide $b_y$ by $b_x$, then let $q$ and $t$ be the quotient and remainder respectively.
    \State Put $b_y\gets b_x, b_x\gets t, t\gets y-qx, y\gets x, x\gets t, z\gets z+1$.
\EndWhile
\If {$z$ is odd}
    \State Put $b_y\gets -b_y$ and $y\gets -y$.
\EndIf
\State Put $a_x\gets Gx$, $a_y\gets Gy$.
\If {$z\neq0$}
    \State Put $c_x\gets \frac{C_y b_x-mx}{B_y}, Q_1\gets b_y c_x, Q_2\gets Q_1+m, d_x\gets \frac{D_y b_x - w_2 x}{B_y}, Q_3\gets y d_x$.
    \State Put $Q_4 \gets Q_3 + D_y, d_y\gets \frac{Q_4}{x}$.
    \If {$b_x\neq 0$}
        \State Put $c_y\gets \frac{Q_2}{b_x}.$
    \Else 
        \State Put $c_y\gets \frac{c_x d_y - w_1}{d_x}.$
    \EndIf
    \State Put $u_3\gets b_y c_y - a_y d_y, w_3\gets b_x c_x - a_x d_x, v_3\gets G(Q_3+Q_4)-Q_1-Q_2$.
\Else 
    \State Put $Q_1\gets C_y b_x, c_x\gets \frac{Q_1-m}{B_y}, d_x\gets \frac{b_x D_y-w_2}{B_y}, u_3\gets b_y C_y, w_3\gets b_x c_x - G d_x, v_3\gets v_2-2Q_1$.
\EndIf
\State Put $t\gets2u_3,A\gets G(xt+yv_3), B\gets Gy, C\gets t$.
\State Put $\gamma\gets \frac{A+B\sqrt{d}}{C}$, and $\varphi_3\gets (u_3,v_3,w_3)$.
\State \Return $\varphi_3,\gamma$.
\end{algorithmic}
\end{algorithm}

For NUDUPL (Algorithm \ref{alg:NUDUPL}), again we precompute $L=\lvert d \rvert^\frac{1}{4}$. As above, except where noted, all divisions are done \textbf{without} remainder.
\begin{algorithm}
\caption{NUDUPL}\label{alg:NUDUPL}
\begin{algorithmic}
\State \textbf{Input:} a binary quadratic form, $\varphi=(u,v,w)$, with discriminant $d$, and the value $L=\lvert d \rvert^\frac{1}{4}$.
\State \textbf{Output:} a near-reduced binary quadratic form $\varphi_3=(u_3,v_3,w_3)$, whose discriminant is $d$, and the value $\gamma\in K$ such that $\varphi_3=\frac{1}{\gamma}\varphi^2$.
\\
\State Use the extended Euclidean algorithm to find $(x,y,G)$ such that $xu+yv=G=\gcd(u,v)$.
\State Put $A_x\gets G,B_y\gets \frac{u}{G}, D_y\gets\frac{v}{G}$.
\State Put $B_x\gets yw \mod B_y$.
\State Put $b_x\gets B_x, b_y\gets B_y, x\gets 1, y\gets 0, z\gets 0$.
\While{$\lvert b_y \rvert > L$ and $b_x\neq 0$}
    \State Divide $b_y$ by $b_x$, then let $q$ and $t$ be the quotient and remainder respectively.
    \State Put $b_y\gets b_x$, $b_x\gets t$, $t\gets y-qx$, $y\gets x$, $x\gets t$, $z\gets z+1$.
\EndWhile
\If{$z$ is odd}
    \State Put $b_y\gets -b_y$ and $y\gets -y$.
\EndIf
\State Put $a_x\gets Gx$, $a_y\gets Gy$.
\If {$z=0$}
    \State $d_x\gets \frac{b_x D_y - w}{B_y}, u_3\gets b_y^2, w_3\gets b_x^2, v_3\gets v-(b_x+b_y)^2+u_3+w_3, w_3\gets w_3-G d_x$.
\Else
    \State Put $d_x\gets \frac{b_x D_y - w x}{B_y}, Q_1\gets d_x y, d_y\gets Q_1+D_y, v_3\gets G(d_y+Q_1), d_y\gets \frac{d_y}{x}, u_3\gets b_y^2$.
    \State Put $w_3\gets b_x^2, v_3\gets v_3-(b_x+b_y)^2+u_3+w_3, u_3\gets u_3-a_y d_y, w_3\gets w_3 - a_x d_x$.
\EndIf
\State Put $t\gets 2u_3, A\gets G(xt+y v_3), B\gets Gy, C\gets t$.
\State Put $\gamma=\frac{A+B\sqrt{d}}{C},\varphi_3=(u_3,v_3,w_3)$.
\State \Return $\varphi_3,\gamma$.
\end{algorithmic}
\end{algorithm}

Our third algorithm is called NUCOMPchoose (Algorithm \ref{alg:NUCOMPchoose}). This algorithm takes two ideals (written as $\BZ$-modules), converts them to binary quadratic forms, and decides which technique (out of NUCOMP, NUDUPL and standard composition without reduction) should be used to compose them. The third technique mentioned, composition without reduction, simply produces the ideal product of its inputs, written as a $\BZ$-module. The reader can see how this is achieved in \cite[Section 5.4]{JacobsonWilliams}. If either ideal has small norm, NUCOMP and NUDUPL tend to produce large values during the computation, which can overflow. It is therefore preferable to use the ideal product in these cases (the authors have found that a norm less than or equal to 50 is appropriate, at least for $d\leq 10^{11}$).

\begin{algorithm}
\caption{NUCOMPchoose}\label{alg:NUCOMPchoose}
\begin{algorithmic}
\State \textbf{Input:} the primitive reduced ideals $I_1=[\frac{u_1}{2},\frac{v_1+\sqrt{d}}{2}]$ and $I_2=[\frac{u_2}{2},\frac{v_2+\sqrt{d}}{2}]$.
\State \textbf{Output:} the primitive (not necessarily reduced) ideal $I=[\frac{u}{2},\frac{v+\sqrt{d}}{2}]$, and the value $\gamma$ such that $I=\frac{1}{\gamma}{I_1 I_2}$.
\\
\State Put $u_1\gets \lvert u_1 \rvert, u_2\gets\lvert u_2 \rvert, v_1\gets v_1\pmod{u_1}, v_2\gets v_2\pmod{u_2}$.
\If{$u_1\leq 50$ or $u_2 \leq 50$}
    \State Find the ideal product of $I_1$ and $I_2$, and write it as $S[\frac{u}{2},\frac{v+\sqrt{d}}{2}]$ for some integer $S$.
    \State \Return $I=[\frac{u}{2},\frac{v+\sqrt{d}}{2}],\gamma=S$.
\EndIf
\State Put $w_1\gets\frac{v_1^2-D}{2u_1}$.
\If{$u_1=u_2$ and $v_1=v_2$}
    \State Use NUDUPL to compose $(\frac{u_1}{2},-v_1,w_1)$ with itself, and let the output be $(u_3,v_3,w_3)$ and $\gamma$.
\Else
    \State Put $w_2\gets\frac{v_2^2-D}{2u_2}$.
    \State Use NUCOMP to compose $(\frac{u_1}{2},-v_1,w_1)$ with $(\frac{u_2}{2},-v_2,w_2)$, and let the output be $(u_3,v_3,w_3)$ and $\gamma$.
\EndIf
\State Put $u=\lvert 2u_3\rvert$ and $v=-v_3\pmod{u}$.
\State \Return $I=[\frac{u}{2},\frac{v+\sqrt{d}}{2}]$ and $\gamma$.
\end{algorithmic}
\end{algorithm}



\clearpage 

\bibliographystyle{plain}
\bibliography{quadraticunits} 



\end{document}